%% file: Rand_Grps_Main.tex
\documentclass[a4paper]{article}
\input{macros}
\input{environs}

\newcommand{\rv}[2][l]{\ensuremath{\mathcal{#2}_{#1}}}
\newcommand{\PP}{\ensuremath{\mathbb{P}}}
\newcommand{\aas}{a.\ a.\ s.}
\newcommand{\dFK}{\ensuremath{d_F(K;\Z{})}}
\title{Virtually abelian quotients of random groups}
\author{Gareth Wilkes}
\begin{document}
\maketitle
\begin{abstract}
It is well-known that random groups with at least as many relators as generators have vanishing first betti number with high probability. In this paper we extend this question and study maps from few-relators random groups to infinite virtually abelian groups.
\end{abstract}
\subsection*{Introduction}
There is no especially natural probability distribution on the class of groups, and accordingly no totally natural definition of what a `random group' should be. In consequence many models of `random groups' are studied, most of which should really be termed `random finite presentations'. One common choice is to fix a free group $F$ of rank at least two and randomly choose a certain number of reduced words $\rho(l)$ of length $l$ in $F$ to be the relators of a `random group' \rv{G}. Other possibilities include models where relators are chosen to be {\em cyclically} reduced; where random elements of the free group of length {\em at most} $l$ are chosen; and where relators are all chosen to be of length 3, but the number of generators of the group is permitted to grow (the {\em triangular model} exploited by \.Zuk \cite{Zuk03}). Ollivier \cite{Ollivier05} has written an excellent introduction to the subject.

Let us set up the notation for the present paper. Let $n>1$ be an integer and let $F=F(X)$ be the free group on a set of $n$ generators $X=\{x_1,\ldots,x_n\}$. For $l\in\N$ let $S_l=S_l(X)$ be the set of reduced words of length $l$ in the alphabet $X$. Let \rv{R} be a subset of $S_l$ formed by randomly choosing $\rho(l)$ elements of $S_l$ uniformly, independently and allowing repetitions, where $\rho(l)$ is some function of $l$. The {\em `random group'} $\rv{G} = \gp{X|\rv{R}}$ satisfies a property $P$ of groups {\em asymptotically almost surely} (\aas) if
\[\PP(\rv{G} \text{ satisfies }P)\to 1 \quad\text{as }l\to \infty\text{.}\]
Corresponding to different choices of $\rho(l)$ one obtains models of groups with different asymptotic properties. The most important choices are the {\em few relators model(s)} introduced by Arzhantseva and Ol'shanskii \cite{AO96}, in which $\rho(l)$ is some constant; and the {\em Gromov density model(s)}, in which $\rho(l)=\flr{|S_l|^d}$ for some `density' $d\in(0,1)$ \cite[Chapter 9]{Gromov93}. The few relators model is sometimes regarded as the `$d=0$ case' of the density model; Kapovich and Schupp \cite{KS08} discuss the relationship between these two models in more detail.

There is a well-known result that a random group on $n$ generators with $n$ relators \aas\ does not admit a surjection to \Z, which may be thought of as follows. The relators in \rv{R} give a family of elements in the abelianisation $\Z^n$ of the free group $F$, which become `well spread out' and therefore `generic' as the length $l$ of the relators grows. Therefore one expects that $\rho(l)=n$ relators should generate a finite index subgroup of $\Z^n$ with high probability, in which case \rv{G} cannot map to \Z{} in a non-trivial way.

We will use this approach to study maps to $\Z$ from subgroups \rv{H} of \rv{G} with a (bounded) finite index. Intuitively one should consider an appropriate finite index subgroup $H$ of the free group $F$, abelianise $H$ and consider the probability that $(H \cap \gpn{\rv{R}})^{\rm ab}$ has finite index in $H^{\rm ab}$, and hence forbids any map $H\to \Z$ descending to a map from \rv{H} to $\Z$. Again the sentiment should be that the relators become spread out sufficiently evenly to force this to happen once the number of relators is sufficiently large. The main result will be the following theorem.
\begin{thmquoteA}
For all $M$ there exists a constant $c(M)$ such that, for any infinite group $G$ having an abelian normal subgroup of index at most $M$,
\[\PP(\exists \text{ a surjection }\rv{G}\twoheadrightarrow G) \to 0\quad \text{ as }l\to\infty \]
for a random group \rv{G} having at least $\rho(l)\geq c(M)$ relators.
\end{thmquoteA}
To give an indication of scale, a somewhat generous upper bound for $c(M)$ is $c(M)\leq Mn$.

We point out that a finite index subgroup of a random group \rv{G} (if any finite index subgroups exist at all) is not generally a `random group' in the same sense as the overgroup \rv{G}, so that this result is not an immediate consequence of the classical result that random groups do not surject to \Z.

Note that there is some flexibility in the statement of this theorem. By factoring out torsion, it suffices to address the case when $G$ has a torsion-free abelian group of index at most $M$. There are only finitely many $n$-genrated groups of this type, so we may equivalently give the conclusion of the theorem as
\[\PP(\exists G \in A_M \text{ admitting a surjection }\rv{G}\twoheadrightarrow G) \to 0 \]
where $A_M$ is the class of infinite groups with an abelian group of index at most $M$.

One somewhat heavy-handed way to capture rigorously the sentiment that $n$ relators in $\Z^n$ become `well spread out' would be to use a Central Limit Theorem to argue that the distribution becomes approximately multivariate normal \cite{FvdH13}, and thus to argue that $n$ points in this distribution are generically a basis. 

In this paper we do not use that viewpoint; instead our analysis is based on studying a certain notion of non-backtracking random walks from the point of view of Markov processes to control the limiting behaviour of the relators. This approach works best on finite state spaces, so instead of maps to \Z{} we will consider maps to finite cyclic groups $\F_q$ for arbitraily large primes $q$.

\begin{ackn}
The author wishes to thank Goulnara Arzhantseva and Federico Vigolo for reading this paper and giving helpful comments. The author was supported by a Junior Research Fellowship from Clare College Cambridge.
\end{ackn}
\begin{cnv} {\hfill}
\begin{itemize}
\item Calligraphic letters will denote objects which are in some sense random.
\item Let $G$ be a group, let $X$ be a set and let $X\to G$ be a function. For a word $w$ in the alphabet $X\cup X^{-1}$ we write $w=_G g$ to mean `the word $w$ evaluates to $g\in G$ under the natural map $F(X)\to G$'.
\end{itemize}
\end{cnv}

\subsection*{Non-backtracking Random Walks}
One subtlety in the formation of random reduced words in a free group is that each letter is not independent of the previous letters, as no letter can follow its inverse. We will manufacture a Markov process from the formation of reduced words by enlarging the state space to include a marker for the last letter added.

Let $G$ be a group, let $X=\{x_1,\ldots, x_n\}$ be a set of size $n\geq 2$ and fix some function $X\to G$. We will tacitly identify the $x_i$ with their images in $G$. Let $\Omega$ be the space $G\times \{\pm1, \pm2,\ldots, \pm n\}$. We will denote an element of $\Omega$ by $(g,\epsilon i)$ where $g\in G$, $\epsilon=\pm 1$ and $i\in\{1,\ldots,n\}$.

A {\em non-backtracking random walk on $(G,X)$} is a Markov process \rv{X} on the space $\Omega$,
where
\[\PP\big(\rv[l+1]{X} = (g,\epsilon i) \mid \rv{X} = (h,\epsilon' j)\big) = \begin{cases}
\alpha_{\epsilon' j, \epsilon i} & \text{if }g=hx_i^{\epsilon} \text{ and } \epsilon i \neq -\epsilon' j \\
0 & \text{otherwise}
\end{cases} \] 
for some positive constants $\alpha_{\epsilon' j, \epsilon i}$ such that \[\sum_{\substack{\epsilon i\ \in\{\pm1,\ldots,\pm n\}\\ \epsilon i\neq -\epsilon' j}} \alpha_{\epsilon' j, \epsilon i} = 1,\]
and where the initial distribution \rv[1]{X} is of the form
\[\PP\big(\rv[1]{X} = (g,\epsilon i)\big) = \begin{cases}
\beta_{\epsilon i} & \text{if }g=x_i^{\epsilon} \\
0 & \text{otherwise}
\end{cases}  \]
for positive constants $\beta_{\epsilon i}$ such that $\sum\beta_{\epsilon i} = 1$. 

We will largely be concerned with the limiting behaviour of these processes for finite groups $G$, and the exact value of the constants $\alpha_{\bullet\bullet}$ and $\beta_\bullet$ will not be especially important. In the cases actually used in this paper---the `unbiased non-backtracking random walk'---one has $\alpha_{\bullet,\bullet}=1/(2n-1)$ and $\beta_\bullet = 1/2n$.

Note that the transition matrix of \rv{X} is invariant under the left-multiplication action of $G$. We may consider $\Omega$ to be a directed graph with an edge from $(h,\epsilon' j)$ to $(g,\epsilon i)$ if and only if
\[\PP\big(\rv[l+1]{X} = (g,\epsilon i) \mid \rv{X} = (h,\epsilon' j)\big)>0\]
Observe that paths in $\Omega$ from $(h,\epsilon' j)$ to $(g,\epsilon i)$ correspond to reduced words $w$ in $F=F(X)$ whose final letter is $x_i^{\epsilon}$ and whose first letter is different from $x_j^{-\epsilon'}$, and which evaluate to $h^{-1}g$ in $G$. In particular loops based at $(g, \epsilon i)$ in $\Omega$ correspond exactly to cyclically reduced words in $F$ whose last letter is $x_i^{\epsilon}$ and which evaluate to 1 in $G$.

Define the associated `summed process' \rv{\overline X} via the projection $\Omega \to G$, so that
\[\PP(\rv{\overline X}=g) = \sum_{\epsilon i} \PP\big(\rv[l]{X} = (g,\epsilon i)\big)\]
Then the probability that a random reduced word of length $l$ in $F$ evaluates to $g$ in $G$ is precisely $\PP(\rv{\overline X}=g)$. 

Note that \rv{\overline X} is not Markov, but will inherit good limiting properties from \rv{X}. Specifically, if \rv{X} converges to some stationary distribution on $\Omega$, then \rv{\overline X} also converges to a limiting `stationary' distribution---which must be the uniform distribution on $G$ by left-invariance.

\begin{rmk}
The reader should note that the nomenclature `non-backtracking random walk on $(G,X)$' may be slightly inaccurate in some cases---specifically if $x_i^2=_G 1$ then \rv{X} may well `backtrack' on a path 
\[(1, +i) \to (x_i, +i) \to (1, +i) \]
More generally, if several generators (or their inverses) happen to be equal in $G$ then the walk \rv{\overline X} on $G$ could also backtrack. However, except in this situation, we would indeed have a non-backtracking random walk on the Cayley (multi)graph of $(G,X)$ and it seems churlish to introduce a new name to acccount for this somewhat degenerate case. 
\end{rmk}
\begin{lem}
The Markov process \rv{X} is irreducible if and only if $G$ is generated by the $x_i$, but is not a free group freely generated by the $x_i$.
\end{lem}
\begin{proof}
Irreducibility of \rv{X} is equivalent to the statement that for any $g\in G$ and any $x_i^{\epsilon}, x_j^{\epsilon'}\in X^{\pm 1}$ there exists a non-trivial reduced word $w\in F=F(X)$ which begins with a letter not equal to $x_j^{-\epsilon'}$, ends with $x_i^{\epsilon}$, and evaluates to $g$ in $G$. This immediately forces $G$ to be generated by the $x_i$ and the case $g=1$ shows that $G$ cannot be freely generated by $X$.

Conversely, suppose that $G$ is generated but not freely generated by the $x_i$. By hypothesis there is some reduced word $r$ in $F$ such that $r=_G 1$. We may construct for each $\epsilon i$ a reduced word $r_{\epsilon i}=_G 1$ which has first letter equal to $x_i^{-\epsilon}$ and final letter equal to $x_i^{\epsilon}$. Any word
\[x_i^{-\epsilon}x_{i+1}^{-1} x_i^{-m} r x_i^m x_{i+1} x_i^{\epsilon}\]
evaluates to $1$, and for $m$ sufficiently large (for example, for $m$ at least the length of $r$) reducing such a word gives an $r_{\epsilon i}$ with the specified initial and terminal letters.

Now take $g\in G$ and $x_i^{\epsilon}, x_j^{\epsilon'}\in X^{\pm 1}$. Since the $x_i$ generate $G$ there is a reduced word $w$ in $F$ such that $w=_G g$. If the final letter of $w$ is not $x_i^{\epsilon}$, then replace $w$ by the reduced word $w r_{\epsilon i}$. If the first letter of this new $w$ is $x_j^{-\epsilon'}$, take $k\neq j$ and replace $w$ by the reduced word $r_{+k} w$. This new $w$ gives a path in $\Omega$ from $(1, \epsilon' j)$ to $(g, \epsilon i)$ and exhibits irreducibility.
\end{proof}
\begin{lem}
Suppose that \rv{X} is irreducible. Then the period of \rv{X} is either 1 or 2. 
\end{lem}
\begin{proof}
By irreducibility, there is a non-trivial reduced word $w$ in $F$ of length $l$ which gives a path in $\Omega$ from $(1, +1)$ to $(1, +1)$---that is, the first letter $x_s^{\epsilon}$ of $w$ is not $x_1^{-1}$ and the last letter is $x_1$. Then $w^3$ is a path in $\Omega$ from $(1, +1)$ to $(1, +1)$ of length $3l$. Consider the following two cases. 

If $s=1$, so that $w$ begins and ends with $x_1$, then the reduced word $wx_2wx_2^{-1} w$ gives a loop of length $3l+2$ in $\Omega$. Hence the period of \rv{X} divides both $3l$ and $3l+2$, hence is at most 2.

If $s\neq 1$, then the words $wx_1x_s^{\epsilon}w x_s^{-\epsilon} x_1^{-1} w$ and $wx_1^2x_s^{\epsilon}w x_s^{-\epsilon} x_1^{-2} w$ provide loops of lengths $3l+4$ and $3l+6$ in $\Omega$, once again showing that the period of \rv{X} is at most 2.
\end{proof}
\begin{lem}
Suppose \rv{X} is irreducible. Then \rv{X} has period 2 if and only if there exists an index 2 subgroup $H$ of $G$ such that $x_i\notin H$ for all $i$.
\end{lem}
\begin{proof}
The `if' direction is clear. Conversely assume that the period of \rv{X} is 2.

Let $H$ be the set of elements $g\in G$ such that there exists an even length reduced word $w$ such that $w=_G g$. Clearly $H$ contains 1 and is invariant under inversion. If $g_1,g_2\in H$ are represented by even length reduced words $w_1$ and $w_2$ then fully reducing $w_1 w_2$ gives an even length reduced word represnting $g_1g_2$. So $H$ is a subgroup of $G$. If $g_1, g_2\in G\smallsetminus H$, then there are odd length words $w_1, w_2$ representing $g_1$ and $g_2$; again reducing $w_1w_2$ gives an even length reduced word showing $g_1g_2\in H$. So $H$ has index 2 in $G$. 

It only remains to show that $x_1\in H$ is impossible. Suppose $x_1\in H$ and consider a reduced word $w$ of even length $2l$ in $F_n$ evaluating to $x_1$ in $G$. By reducing and cyclically reducing $x_1^{-1}w$ we find an odd-length loop in $\Omega$, which is forbidden since the period of \rv{X} is 2.
\end{proof}
Note that the subgroup $H$ in the above lemma is unique: if $H_1$ and $H_2$ are index 2 subgroups of $G$ such that no $x_i$ lies in $H_1\cup H_2$, then the image of the natural quotient map $G\to G/H_1\times G/H_2$ has order 2, so that $H_1=H_2$.
\begin{lem}
Suppose \rv{X} is irreducible and has period 2, and let $H$ be the subgroup from the previous lemma. Consider the processes $\rv{Y}=\rv[2l+2]{X}$ and $\rv{Z}=\rv[2l+1]{X}$ on $\Omega$. Then \rv{Y} is an irreducible aperiodic Markov process on $\Omega_H= H\times\{\pm1, \pm2,\ldots, \pm n\}$ and \rv{Z} is an irreducible aperiodic Markov process on $\Omega'_H=(G\smallsetminus H)\times\{\pm1, \pm2,\ldots, \pm n\}$.
\end{lem}
\begin{proof}
By the definition of the initial distribution of \rv{X}, we find that $\rv{Y}$ lies in $\Omega_H $ for all $l$ and $\rv{Z}$ lies in $\Omega'_H$ for all $l$. 

Irreducibility of \rv{Y} is the statement that any two points in $\Omega_H$ are connected in $\Omega$ by an even length path---but they are connected by some path by irreducibility of \rv{X}, and any odd length path must, by definition of $H$, connect a point in $\Omega_H$ to $\Omega'_H$ or vice versa. Hence \rv{Y} is irreducible. Similarly \rv{Z} is irreducible.

Finally note that the transition matrices for \rv{Y} and \rv{Z} are identical, so they have the same period. But the period of \rv{X} is, by definition, twice the greatest common divisor of these two periods. Hence \rv{Y} and \rv{Z} are aperiodic and the lemma is complete.
\end{proof}
\begin{theorem}\label{ConvOfRWs}
Let $G$ be a finite group, let $X=\{x_1,\ldots, x_n\}$ be a set of size $n\geq 2$ and fix a function $X\to G$. Suppose $G$ is generated by $X$ and consider a non-backtracking random walk \rv{X} on $(G,X)$. If \rv{X} is aperiodic then \rv{\overline X} converges (in distribution) to the uniform distribution on $G$. Otherwise there exists an index 2 subgroup $H$ of $G$ with $x_i\notin H$ for all $i$ such that $\rv[2l]{\overline X}$ converges to the uniform distribution on $H$ and $\rv[2l+1]{\overline X}$ converges to the uniform distribution on $G\smallsetminus H$.
\end{theorem}
\begin{proof}
Given the previous results, we find by the standard theory of Markov chains that the processes in the theorem converge to some stationary distributions. It remains to note that these are uniform on the required sets.

When \rv{X} is aperiodic, it converges to the unique stationary distribution $\pi_{(g,\epsilon i)}$ on $\Omega$ defined by the relation
\[ \pi_{(g,\epsilon i)} = \sum_{h,\epsilon' j} \pi_{(h,\epsilon'j)}p_{(h,\epsilon' j),(g,\epsilon i)} \]
where the $p_{(h,\epsilon' j),(g,\epsilon i)}$ are the transition probabilities
\[ p_{(h,\epsilon' j),(g,\epsilon i)}=\PP\big(\rv[l+1]{X} = (g,\epsilon i) \mid \rv{X} = (h,\epsilon' j)\big)\]
These transition probabilities are left $G$-invariant by definition, hence for any $g_0\in G$ the translated distribution $g_0\cdot \pi_{(g,\epsilon i)} = \pi_{(g_0g,\epsilon i)}$ also satisfies the above condition, and hence equals $\pi$---that is, the stationary distribution of \rv{X} on $\Omega$ is $G$-invariant. Hence the limiting distribution of \rv{\overline X} on $G$ is also $G$-invariant, and hence uniform.

The uniformity of \rv{Y} and \rv{Z} on the given sets follows similarly, given the fact that the two-step transition probabilities are also $G$-invariant.   
\end{proof}

\subsection*{Virtually abelian quotients}
We will now study the maps to virtually abelian groups in the few relators model of random groups, with the following goal in mind.
\begin{theorem}\label{VirtAbelian}
For all $M\in\N$ there exists a constant $c(M)$ such that, for any infinite group $G$ having an abelian normal subgroup of index at most $M$,
\[\PP(\exists \text{ a surjection }\rv{G}\twoheadrightarrow G) \to 0\quad \text{ as }l\to\infty \]
for a random group \rv{G} having at least $\rho(l)\geq c(M)$ relators.
\end{theorem}
Note one may remove the assumption of normality by replacing the constant $c(M)$ by $c(M!)$, as the normal core of a subgroup of index at most $M$ has index at most $M!$.

Since there are only finitely many groups of a given order up to isomorphism, it suffices to fix a  finite group $J$ of order at most $M$ and to prove the theorem for those $G$ which are an extension of $J$ by an infinite abelian group $A$. We may also without loss of generality assume that $A$ is free abelian: the torsion subgroup of $A$ is characteristic, hence is normal in $G$ and may be factored out. Finally, since there are only finitely many maps from $F$ to $J$, it suffices to consider maps to $G$ which {\em carry} some fixed map $f\colon F\to J$ in the sense that the composition $F\to \rv{G} \to G \to J$ equals $f$. 

We will approach the theorem by `approximating' $G$ using its finite quotients $G/qA$ by the characteristic subgroups $qA$ of $A$, where $q$ is a large prime. We will place a bound on the limiting probability that \rv{G} admits a surjection to $G/qA$; allowing $q$ to tend to infinity will then yield the result. 

In the case when $J$ is trivial, this bound would be obtained in the following manner: a map from $\rv{G}$ to the $\F_q$-vector space $A/qA$ exists precisely when the images of the relators \rv{R} in $H_1(F;\F_q)$ fail to generate $H_1(F;\F_q)\iso\F_q^n$. The fact that the distribution of the images of the relators becomes eventually uniform, combined with elementary estimates of the number of $n$-tuples which generate $\F_q^n$, bounds the probability that \rv{G} can map to $A/qA$.

When $J$ is non-trivial, we will follow a similar procedure: however the natural object to consider is no longer simply a vector space $H_1(F;\F_q)$. Instead we must consider $H_1(\ker f; \F_q)$ with its natural conjugation action of $F/\ker f = J$. To study this $J$-module (that is, representation of $J$ over $\F_q$) we will appeal to some representation theory. In particular we will use the fact that this module is semisimple for $q$ coprime to $|J|$ to bound the number of tuples which generate $H_1(\ker f; \F_q)$ over $J$.
\begin{lem}
Take a finite group $J$ and a surjection $f\colon F\to J$. Let $K$ be the kernel of $f$ and let $m=\dFK$ be the minimal number of generators of $K^{\rm ab}$ as a $J$-module. Let $q>|J|$ be a prime. Let $E$ be an irreducible representation of $J$ over the field $\F_q$ and let $G$ be an extension of groups
\[1\to E\to G\to J\to 1 \]
such that conjugation in $G$ induces the given action of $J$ on $E$. 

Then for all $\epsilon >0$ there exists an $L$ such that for $l\geq L$,
\[\PP(\exists \phi\colon\rv{G} \twoheadrightarrow G\text{ carrying }f)\leq \left(1- (1-\epsilon)^m\prod_{j=1}^m (1 - \frac{1}{|E|^j} ) \right)\left( \frac{2+2\epsilon}{|J|}\right)^m \]
for a random group \rv{G} having at least $m$ relators.
\end{lem}
\begin{rmk}
The indefinite article in the phrase `an extension of groups' above may in fact be replaced by a definite article---every such extension is isomorphic to the split extension $E\rtimes J$ since $|J|$ and $|E|$ are coprime, so that $H^2(J,E)=0$.
\end{rmk}
\begin{proof}
Let $K'=H_1(K;\F_q)$ be the modulo $q$ abelianisation of $K$. The action of $F$ on $K$ by conjugation reduces to an action of $J$ on $K'$, and by definition the rank of $K'$ over $J$ is at most $m$. By Maschke's Theorem \cite[Theorem 4.1.1]{Kow14} $K'$ is a semisimple $J$-module. 

Let $K'_E$ be the $E$-isotypic component of $K'$, and let $\pi_E\colon K' \to K'_E$ be the projection map. Note that $K'_E$ is equal to a direct sum of at most $m$ copies of $E$ and that any $J$-linear map $K'\to E$ factors through $\pi_E$. 

Suppose there exists a map $\phi\colon \rv{G}\to G$ carrying $f$. Then $f\colon F\to J$ descends to a map $\rv{G} \to J$, so the relators \rv{R} of \rv{G} lie in $K$. Furthermore the restriction of $\phi$ to a map $K \to E$ yields a non-trivial $J$-linear map $K'\to E$ which vanishes on the image of \rv{R} in $K'$. It follows that $\pi_E(\rv{R})$ does not generate $K'_E$ as a $J$-module. (Conversely if the relators lie in $K$ but do not generate $K'_E$, then they lie in some direct summand of $K'_E$ so there is a surjection from \rv{G} to the extension $G$ of $J$ by $E$.)

Now consider the short exact sequence
\[1\to K'_E \to H \to J \to 1 \]
where $H$ is the quotient of $F/[K,K]K^q$ by $\ker\pi_E$. Forming reduced words in $F$ now gives a non-backtracking random walk \rv{X} on $(H,X)$. By Theorem \ref{ConvOfRWs}, if \rv{X} is aperiodic then the summed process \rv{\overline X} converges to the uniform distribution. So, for all $\epsilon>0$, there exists $L$ such that for $l\geq L$ we have 
\[\big|\PP(\rv{\overline X} = h\,|\,\rv{\overline X}\in B)-1/|B|\big|\leq \epsilon/|B|\] 
for each subset $B\subseteq H$ and all $h\in B$. 

Now take $m$ instances $\rv{\overline X}^{(i)}$ of this process corresponding to the first $m$ relators of $\rv{G}$. We have
\begin{eqnarray*}
\lefteqn{\PP(\exists \rv{G}\twoheadrightarrow G\text{ carrying }f)} \\
 & \leq& \PP\big(\rv{\overline X}^{(1)},\ldots,\rv{\overline X}^{(m)}\text{ do not generate }K'_E \,\big|\, \rv{\overline X}^{(i)} \in K'_E \forall i\big) \cdot \PP(\rv{\overline X}^{(i)} \in K'_E \forall i)\\
 & \leq & \left( 1- \PP\big(\rv{\overline X}^{(1)},\ldots,\rv{\overline X}^{(m)}\text{ do generate }K'_E \,\big|\, \rv{\overline X}^{(i)} \in K'_E \forall i\big)\right)\left( \frac{1+\epsilon}{|J|}\right)^m \\
 &\leq & \left(1- (1-\epsilon)^m\prod_{j=1}^m (1 - \frac{1}{|E|^j} ) \right)\left( \frac{1+\epsilon}{|J|}\right)^m
\end{eqnarray*}
where the final line uses the formula
\[\big| \{(y_1,\ldots y_m)\in (E^{m})^m \text{ which generate }E^{m}\text{ over }J \}\big| = \prod_{j=1}^m (1 - \frac{1}{|E|^j} ) \cdot |E|^{m^2}\]
completely analogous to the computation of $|{\rm GL}_m(\F_q)|$.

On the other hand, if \rv{X} has period 2 then by Theorem \ref{ConvOfRWs} there is some index 2 subgroup $J'$ of $J$ such that $\rv[l]{X}\in J'$ if and only if $l$ is even; and the distribution of \rv[2l]{X} converges to the uniform distribution on $J'$. Note that $K'_E\subseteq J'$. The probability that there exists a map $\rv{G}\twoheadrightarrow G$ carrying $H$ is now zero when $l$ is odd; and when $l$ is even we may perform essentially the same calculation as above, but with $|J|$ replaced by $|J'|=|J|/2$.
\end{proof}
\begin{prop}\label{FixedVirtAbelian}
Take a finite group $J$ and a surjection $f\colon F\to J$. Let $K$ be the kernel of $f$ and let $m=\dFK$. Let $G$ be an extension of groups
\[1\to A\to G\to J\to 1 \]
such that $A$ is a (non-trivial) free abelian group. Then 
\[\PP(\exists \phi\colon\rv{G} \twoheadrightarrow G\text{ carrying }f) \to 0\quad \text{ as }l\to\infty \]
for a random group \rv{G} having at least $m$ relators.
\end{prop}
\begin{proof}
Suppose the contrary, so that there exist a $\delta>0$ and an infinite sequence $(l_k)_{k\geq 1}$ such that for all $k$, 
\[\PP(\exists \phi\colon\rv[l_k]{G} \twoheadrightarrow G\text{ carrying }f)\geq\delta \]
Now, if such a map $\phi\colon \rv[l_k]{G} \to G$ exists, then for all primes $q>|J|$ there exists a map $\rv[l_k]{G} \twoheadrightarrow G/qA$. Furthermore $A/qA$ is a semisimple $J$-module by Maschke's theorem, hence there is an irreducible $J$-module $E$ and a quotient group $G_E$ of $G/qA$ which is an extension of $J$ by $E$. Taking $l_k$ sufficiently large, we may apply the previous lemma for $\epsilon = 1/2$. Then for each $q$ we have, for some $l_k$ sufficiently large,
\begin{eqnarray*}
\delta\leq \PP(\exists\rv[l_k]{G} \twoheadrightarrow G\text{ carrying }f)&\leq&  \PP(\exists \rv[l_k]{G} \twoheadrightarrow G_E\text{ carrying }f\text{ for some }E)\\
&\leq & \sum_{|E|} \left(1- \frac{1}{2^m}\prod_{j=1}^m (1 - \frac{1}{|E|^j} ) \right)\left( \frac{3}{|J|}\right)^m\\
&\leq& |J| \left(1- \frac{1}{2^m}\prod_{j=1}^m (1 - \frac{1}{q^j} ) \right)\left( \frac{3}{|J|}\right)^m
\end{eqnarray*}
where in the final line we have used the fact that for all $q$ there are at most $|J|$ distinct irreducible $J$-representations over the field $\F_q$ \cite[Theorem 4.2.5]{Kow14}. The quantity on the right hand side of this equation tends to zero as $q\to \infty$---in particular there is $q$ such that it is less than $\delta$, giving a contradiction. This proves the proposition.
\end{proof}
\begin{proof}[Proof of Theorem \ref{VirtAbelian}]
Given $M$, fix a constant number of relators
\[c(M) = \max\left\{d_F(K)\mid K\nsgp F\text{ and }[F:K]\leq M \right\} \]
The theorem now follows immediately from Proposition \ref{FixedVirtAbelian}. Note that $c(M)\leq Mn$.
\end{proof}
We note the following equivalent characterisation of Theorem \ref{VirtAbelian}.
\begin{theorem}
For all $M\in\N$ there exists a constant $c'(M)$ such that 
\[\PP\big(\exists H\leq\rv{G} \text{ such that }[G:H]\leq M\text{ and }H\twoheadrightarrow \Z \big) \to 0 \quad \text{ as }l\to\infty \]
for a random group \rv{G} having at least $\rho(l)\geq c'(M)$ relators.
\end{theorem}
\begin{proof}
Whenever \rv{G} admits such a subgroup $H$ we may consider the normal core $H'$ of $H$ in \rv{G}, which has index at most $M!$ in $\rv{G}$ and admits a surjection to \Z. Then \rv{G} admits a surjection to an infinite subgroup of the wreath product $\Z \wr (\rv{G}/H')$, which has an abelian normal subgroup of index at most $M!$. Applying Theorem \ref{VirtAbelian} with $c'(M)= c(M!)$ relators gives the result.
\end{proof}
It is interesting to compare this to other results concerning random groups. Theorem \ref{VirtAbelian}, concerning constant numbers of relators, tends to suggest to the unwary that in a model with many relators---for example in the density model at density $d$---one could expect or hope that every finite index subgroup has first betti number zero. This is however false in full generality: random groups at density $d<1/6$ are known to be hyperbolic and to act freely and cocompactly on CAT(0) cube complexes \cite{OW11}, hence are virtually special \cite{Agol13} and certainly have positive virtual first betti number. At densities $d>1/3$, \.Zuk's theorem \cite{Zuk03, KK13} applies and random groups have property (T), hence have virtual first betti number zero. It appears to be unknown where the phase transition lies.

\bibliographystyle{alpha}
\bibliography{Rand_Grps.bib}
\end{document}

%% file: macros.tex
\usepackage{amsmath}
\usepackage{amsfonts}
\usepackage{amssymb}
\usepackage{amsthm}
\usepackage{mathtools}
\usepackage{tikz}
\usepackage{tikz-cd}
\usetikzlibrary{patterns}

\newcommand{\iso}{\ensuremath{\cong}}
\newcommand{\Z}[1][]{\ensuremath{\mathbb{Z}_{#1}}}

\newcommand{\N}{\ensuremath{\mathbb{N}}}

\newcommand{\F}{\ensuremath{\mathbb{F}}}

\newcommand{\nsgp}[1][]{\ensuremath{\triangleleft_{#1}}}

\newcommand{\gp}[1]{\ensuremath{\langle #1\rangle}}
\DeclarePairedDelimiter\gpn{\langle\! \langle}{\rangle\! \rangle}

\DeclarePairedDelimiter\flr{\lfloor}{\rfloor}

%% file: environs.tex
\newtheorem{theorem}{Theorem}[]

\newtheorem{prop}[theorem]{Proposition}

\newtheorem{lem}[theorem]{Lemma}

\theoremstyle{definition}

\newtheorem*{cnv}{Notations}
\newtheorem*{ackn}{Acknowledgements}
\theoremstyle{remark}
\newtheorem*{rmk}{Remark}

\theoremstyle{plain}
\newtheorem*{thmquoteA}{Theorem \ref{VirtAbelian}}
\newcounter{introthmcount}
\theoremstyle{definition}